\documentclass[12pt, reqno, twoside]{amsart}
\usepackage{amsmath, amsthm, mathrsfs, amscd, amsfonts, amssymb, graphicx, color}
\usepackage[bookmarksnumbered, colorlinks, plainpages]{hyperref}
\hypersetup{colorlinks=true,linkcolor=red, anchorcolor=green, citecolor=cyan, urlcolor=red, filecolor=magenta, pdftoolbar=true}
\usepackage{enumitem}
\usepackage{enumitem}
\newtheorem{theorem}{Theorem}[section]

\theoremstyle{definition}

\theoremstyle{remark}
\newtheorem{remark}[theorem]{Remark}
\numberwithin{equation}{section}

\begin{document}

\setcounter{page}{1}

\title[Nonlinear Diffusion, Geometric and Functional Inequalities]{Nonlinear Diffusion, and Geometric and Functional Inequalities on Smooth Metric Measure spaces}
\author[A. Taheri]{Ali Taheri}
 \address{School of Mathematical and  Physical Sciences, 
University of Sussex, Falmer, Brighton, United Kingdom.}
\email{\textcolor[rgb]{0.00,0.00,0.84}{A.Taheri@sussex.ac.uk}}




\begin{abstract} 
This extended abstract is based on a talk given at the workshop and summer school ``Direct and Inverse Problems with 
Applications" in Ghent Analysis and PDE Centre in August 2024. It focuses on nonlinear diffusion equations of slow and 
fast types and their links with some geometric and functional inequalities in the framework of smooth metric measure 
spaces. The article presents some introduction in a summer school style as well as several new results.  
\end{abstract}
\maketitle

\section{Introduction and Preliminaries}
\setcounter{equation}{0}
\label{section ross}

Inequalities between some notion of ``{\it size}" of, sets and their perimeters, functions and their gradients, or more generally, 
tensor fields and their differentials, covariant or Lie derivatives, play an important role in geometric analysis. I will give an 
introduction to certain aspects of the subject (without proofs) and describe some new research directions emerging from 
the need to extend known inequalities from Euclidean to the framework of (smooth) metric measure spaces 
({\it see}, e.g., \cite{AGS, AMS, Bak, Otto, Ruz-Sur, Taheri-book-one, Taheri-book-two, VC, Wang, Zhang}).

\section{Isoperimetric inequality and the Isoperimetric profile}

Recall that the classical isoperimetric inequality in $\mathbb R^n$ ($n \ge 2$) asserts that for any sufficiently smooth 
compact subdomain $\Omega \subset \mathbb R^n$:  
\begin{equation}
n \omega_n^{1/n} [{\rm Vol}(\Omega)]^{1-1/n} \le {\rm Per}(\Omega).
\end{equation} 
Here ${\rm Vol}(\Omega)={\mathscr L}^n(\Omega)$ denotes the volume of $\Omega$, 
${\rm Per}(\Omega)={\mathscr H}^{n-1}(\partial \Omega)$ denotes the perimeter of $\Omega$ and $\omega_n$ 
is the volume of the unit $n$-ball. This inequality can equivalently be written (with ${\mathbb B}={\mathbb B}_n$ 
denoting the unit $n$-ball) in the more suggestive form 
\begin{equation} \label{isop-Rn-equation}
\frac{{\rm Per}(\Omega)}{[{\rm Vol}(\Omega)]^\frac{n-1}{n}} \ge \frac{{\rm Per}({\mathbb B})}{[{\rm Vol}({\mathbb B})]^\frac{n-1}{n}}
= \frac{n \omega_n}{\omega_n^\frac{n-1}{n}} = n \omega_n^\frac{1}{n}.
\end{equation}

The inequality asserts that for fixed volume, balls {\it minimise} perimeter, or for fixed perimeter, balls {\it maximise} 
volume. There are various ways to prove this inequality, but the most common and quickest one is via 
the Brunn-Minkowski inequality: 
\begin{equation}
[{\rm Vol}(A+B)]^\frac{1}{n} \ge [{\rm Vol}(A)]^\frac{1}{n} + [{\rm Vol}(B)]^\frac{1}{n}, 
\end{equation}
where $A, B \subset \mathbb R^n$ are non-empty compact sets and $A+B=\{a+b \in \mathbb R^n : a \in A, b \in B\}$. 
Another way of writing \eqref{isop-Rn-equation} is to introduce for $v \ge 0$ the isoperimetric profile function 
$I(v) = \inf \{ {\rm Per}(\Omega) : {\rm Vol}(\Omega)=v\}$. Then ${\rm Per}(\Omega) \ge I({\rm Vol}(\Omega))$ 
where $I(v)= n \omega_n^{1/n} v^{1-1/n}$. Moreover, for any $n$-ball ${\mathcal B}$ we have 
$I({\rm Vol}({\mathcal B})) = {\rm Per}({\mathcal B})$.

\qquad \\
{\bf The Levy-Gromov isoperimetric inequality.} A first natural question here is if the isoperimetric inequality has 
a counterpart for Riemannian manifolds, and if so, what is it? Starting with the cases of the round unit 
$n$-sphere ${\mathbb S}^n$, and the real hyperbolic $n$-space ${\mathbb H}^n$ (i.e., model spaces of constant 
sectional curvatures $\pm1$) we have the following, which are in spirit quite close to (and complement) what was 
observed above for ${\mathbb R}^n$: 
\begin{itemize}
\item Given $\Omega \subset {\mathbb S}^n$, if ${\mathcal B} \subset {\mathbb S}^n$ is a geodesic ball 
(i.e., a spherical cap) such that ${\rm Vol}(\Omega)={\rm Vol}({\mathcal B})$ then 
${\rm Per}(\Omega) \ge {\rm Per}({\mathcal B})$. 
\item Given $\Omega \subset {\mathbb H}^n$ with compact closure, if ${\mathcal B} \subset {\mathbb H}^n$ is a geodesic ball 
such that ${\rm Vol}(\Omega)={\rm Vol}({\mathcal B})$ then ${\rm Per}(\Omega) \ge {\rm Per}({\mathcal B})$. 
\end{itemize}

Returning now to the original question, the answer for compact manifolds without boundary 
(i.e., closed manifolds) is in the affirmative and is provided by the celebrated Levy-Gromov isoperimetric inequality: 
\begin{itemize}
\item Let $({\mathscr M}, g)$ be a closed Riemannian manifold 
with ${\rm dim}({\mathscr M}) =n \ge 2$, ${\mathscr Ric}({\mathscr M}) \ge n-1 [={\mathscr Ric}({\mathbb S}^n)]$. 
Let $\Omega \subset {\mathscr M}$ be a bounded smooth subdomain and let 
${\mathcal B} \subset {\mathbb S}^{n}$ be a round geodesic open ball (spherical cap). Then we have the implication 
\begin{equation}
\frac{{\rm Vol}(\Omega)} {{\rm Vol}({\mathscr M})} = \frac{{\rm Vol}({\mathcal B})}{{\rm Vol}({\mathbb S}^{n})} 
\implies 
\frac{{\rm Per}(\Omega)} {{\rm Vol}({\mathscr M})} \ge \frac{{\rm Per}({\mathcal B})}{{\rm Vol}({\mathbb S}^{n})}.
\end{equation}
Here ${\mathscr M}$ can be {\it any} closed manifold and the result compares the isoperimetric nature of 
${\mathscr M}$ to that of an associated model space (of constant curvature).
\end{itemize}

This discussioin can be conveniently reformulated in terms of an object called 
the isoperimetric profile (of the Riemannian manifold). To this end, for $ 0 \le v \le 1$ let  
\begin{equation}
I_{({\mathscr M}, g)}(v) = \inf 
\left\{ \frac{{\rm Per}(\Omega)}{{\rm Vol}({\mathscr M})} : \frac{{\rm Vol}(\Omega)} {{\rm Vol}({\mathscr M})} =v \right\}. 
\end{equation}
Then the Levy-Gromov isoperimetric inequality is the statement that $I_{({\mathscr M}, g)} \ge I_{({\mathbb S}^n, g_{{\mathbb S}^n})}$.

\section{The Gauss Space $(\mathbb R^n, g_{{\mathbb R}^n}, d\gamma_n)$}
\label{sec-Gaussian}

The Gauss space is the most fundamental and natural example of a smooth metric measure space (that we will come to shortly). 
It consists of the Euclidean space ${\mathbb R}^n$ ($n\ge1$) equipped with its canonical 
metric $g_{{\mathbb R}^n}$ and the Gaussian measure $d\gamma=d\gamma_n$ 
\begin{equation} \label{Gaussian-measure}
d\gamma_n = \prod_{j=1}^n \frac{e^{- \frac{1}{2} x_j^2}}{\sqrt{2\pi}} dx_j 
= \frac{e^{-\frac{1}{2} |x|^2}}{(2\pi)^\frac{n}{2}} dx.
\end{equation}

In the study of isoperimetric inequality for Gauss space, {\it half-spaces} play an important role. For unit vector 
$u \in {\mathbb R}^n$ and $a \in \mathbb R$, let ${\mathscr H}_a = \{x \in {\mathbb R}^n : \langle x, u \rangle \le a\}$. 
Moreover, for any Borel measurable set $A \subset {\mathbb R}^n$ set 
\begin{equation}
\gamma_n^+ (A) = \liminf_{\varepsilon \searrow 0} \frac{\gamma_n(A_\varepsilon) - \gamma_n(A)}{\varepsilon}, 
\qquad \gamma_n(A) = \int_{A} d \gamma_n, 
\end{equation}
where for $\varepsilon>0$ we write $A_\varepsilon = \{x \in {\mathbb R^n | {\rm dist}(x, A) < \varepsilon}\}$. 
Here $\gamma_n^+$ is called Minkowski's exterior boundary measure (or content or simply perimeter). Then we have
\begin{equation}
\gamma_n(A) = \gamma_n({\mathscr H}_a) \implies \gamma_n^+(A) \ge \gamma_n^+({\mathscr H}_a).
\end{equation}
Thus among all sets of a fixed Gaussian measure, half-spaces have the least perimeter.

As with the Levy-Gromov case, one can introduce the isoperimetric profile by writing, 
$I(v) = \inf \{\gamma_n^+ (A) : \gamma_n(A) =v\}$ for $v \in [0,1]$. The Gaussian 
isoperimetric inequality then asserts that
\begin{equation} \label{Gaussian-isoperimetric-inequality}
I(\gamma_n(A)) \le \gamma_n^+(A).
\end{equation}
Here one can also show that (regardless of $n$) $I = \phi \circ \Phi^{-1} : [0,1] \to [0, 1/\sqrt {2\pi}] $ where 
\begin{equation}
\Phi(r) = \gamma_1 (-\infty, r) = \int_{-\infty}^r \frac{e^{-\frac{1}{2} x^2}}{\sqrt{2\pi}} dx, \qquad \phi=\Phi'.
\end{equation}
Note that for half-spaces we have $\Phi(a)=\gamma_n({\mathscr H}_a)$ and $\phi(a)=\gamma_n^+({\mathscr H}_a)$.

\qquad \\
{\bf Functional inequalities on $(\mathbb R^n, g_{{\mathbb R}^n}, d\gamma_n)$.} With this introduction on isoperimetric 
inequalities, and still within the context of the Gauss space $(\mathbb R^n, g_{{\mathbb R}^n}, d\gamma_n)$, we present 
two important functional inequalities, that will serve as motivation for the further development of this article. These are in 
turn the Poincar\'e inequality (also called the spectral gap inequality for the Hermite polynomials) and the logarithmic 
Sobolev inequality. (See, e.g., \cite{Bak, VC} and the references therein for more.)

\begin{itemize}
\item The first inequality takes the form:  
\begin{equation} \label{Gaussian-Poincare}
\int_{\mathbb R^n} f^2 \, d\gamma_n 
- \left( \int_{\mathbb R^n} f \, d\gamma_n \right)^2 \le \int_{\mathbb R^n} |\nabla f|^2 d\gamma_n.  
\end{equation}
\item The second inequality, i.e., the logarithmic Sobolev inequality takes the form:  
\begin{equation} \label{Gaussian-LSI}
\int_{\mathbb R^n} |f|^2 \log |f|^2 \, d\gamma_n 
- \int_{\mathbb R^n} |f|^2 \, d\gamma_n \log \int_{\mathbb R^n} |f|^2 \, d\gamma_n 
\le 2 \int_{\mathbb R^n} |\nabla f|^2 \, d\gamma_n, 
\end{equation}
or equivalently 
\begin{equation}
\int_{\mathbb R^n} |f|^2 \log \frac{|f|^2}{||f||^2_{L^2(d\gamma_n)}} \, d\gamma_n 
\le 2 \int_{\mathbb R^n} |\nabla f|^2 \, d\gamma_n.
\end{equation}
\end{itemize}

Logarithmic Sobolev inequalities are well-known to be of huge significance and to have numerous applications. 
In the context described above, one important implication is the following hypercontractivity of the Nelson 
semigroup $P_t=e^{tN}$ ($1<p <2$): \footnote{One interesting application of this hypercontractivity is to polynomial chaos.}
\begin{equation*}
||e^{tN} f||^2_{L^2(d\gamma_n)} = \int_{\mathbb R^n} |e^{tN} f|^2 \, d\gamma_n \le ||f||^2_{L^p(d\gamma_n)}, \qquad e^{-t} \le \sqrt{p-1}.
\end{equation*}

In relation to the discussion on logarithmic Sobolev inequalities above there two useful points to remember. 
\begin{itemize}
\item Firstly, the function $r \mapsto \phi(r) = \log ||f||_{L^{1/r}(d\mu)}$ for $0<r<1$ is convex.  
\item Secondly, the derivative of the $L^p$-norm with respect to $p$ is given by the formula:  
\begin{align} \label{measure-norm-p-derivative}
\frac {d}{dp} \|f\|_{L^p(d\mu)} 
&= \frac{1}{p^2} \|f\|_{L^p(d\mu)}^{1-p}
\left[ \int_{\mathscr M} |f|^p \log |f|^p \, d\mu - \int_{\mathscr M} |f|^p \, d\mu \log \int_{\mathscr M} |f|^p \, d\mu \right] \nonumber \\
&= \frac{1}{p^2} \|f\|_{L^p(d\mu)}^{1-p} \int_{\mathscr M} |f|^p \log \frac{|f|^p}{||f||^p_{L^p(d\mu)}} \, d\mu. 
\end{align}
\end{itemize}
{\bf Notation.} We make use of the following notations for Energy, Fisher Information, Variance and Entropy 
of a function $f$ on $({\mathscr M}, g, d\mu)$. Here we assume $f>0$. Moreover, later we often take $d\mu$ 
a probability measure, although for the sake of these notations, this restriction is not necessary.
\begin{equation}
{\rm E}_\mu (f) = \int_{\mathscr M} |\nabla f|^2 \, d\mu, \qquad I_\mu(f) = \int_{\mathscr M} \frac{|\nabla f|^2}{f} \, d\mu,  
\end{equation}
\begin{equation}
{\rm Var}_\mu (f) = \int_{\mathscr M} f^2 \, d\mu - \left( \int_{\mathscr M} f \, d\mu \right)^2,
\end{equation}
\begin{align}
{\rm Ent}_\mu (f) &= \left[ \int_{\mathscr M} f \log f \, d\mu - \int_{\mathscr M} f \, d\mu \log \int_{\mathscr M} f \, d\mu \right].
\end{align}
The variance and entropy are easily seen to be non-negative, when $d\mu$ is a probability measure, by the use of Jensen's 
inequality. (For entropy use the convex function $\phi(t)=t \log t$ for $t>0$.) With the notations above, 
we can re-write \eqref{Gaussian-Poincare} and \eqref{Gaussian-LSI} as 
\begin{equation}
{\rm Var}_{\gamma_n}(f) \le {\rm E}_{\gamma_n}(f), \qquad  
{\rm Ent}_{\gamma_n} (f^2) \le 2 {\rm E}_{\gamma_n} (f), 
\end{equation} 
respectively. Furthermore we can re-write \eqref{measure-norm-p-derivative} as 
\begin{align} 
\frac {d}{dp} \|f\|_{L^p(d\mu)} = \frac{1}{p^2} \|f\|_{L^p(d\mu)}^{1-p} {\rm Ent}_\mu(|f|^p).
\end{align}

\qquad \\
{\bf Geometric and Functional Inequalities on $\mathbb R^n$.} It is natural and useful to list here (for purely 
motivational purposes) {\it some} of the geometric and functional inequalities on ${\mathbb R}^n$. 
Each of these inequalities underlines, signifies or links to a specific feature of ${\mathbb R}^n$. We also 
refrain from writing the word ``{\it inequality}" each time for brevity. These are: Nash, Hardy, Poincar\'e, 
Faber-Krahn, Isoperimetric, Loomis-Whitney, Brascamp-Lieb, Hausdorff-Young, Logarithmic Sobolev, 
Caffarelli-Kohn-Nirenberg, Gagliardo-Nirenberg-Sobolev, Littlewood-Paley $g$-function, 
Hardy-Littlewood-Wiener maximal function, Isocapacitory, measure-capacity to name a few. 
(For more, {\it see} \cite{AGS, AMS, CavMon-1, CavMon-2, Milman, Ruz-Tur, Ruz-Sur, Taheri-book-one, Taheri-book-two, VC, Zhang} and the references therein.)

\qquad \\
{\bf Stages of Generalisations.}
In passing from ${\mathbb R}^n$ to more general spaces, depending on the algebraic/geometric structure 
of the space involved or the technicalities available, one can present the following list:  
\footnote{The first two on the list are generalisations of ${\mathbb R}^n$ with a view towards maintaining 
the harmonic analysis tools. In the first case, one can essentially define the Fourier transform (Gelfand 
transform, characters and Pontryagin duality). In the second case, the situation is much more complicated, but 
there are more advanced tools of a similar nature available (Representation theory, Peter-Weyl theory, 
Harish-Chandra theory, Fourier-Helgason transform and Plancherel measure to name a few). As for the 
others, the set of techniques is vast. In the last case the theory of optimal transport is a key tool. }
\begin{itemize}
\item[$\bullet$]  Locally Compact Abeilan Groups 
\item[$\bullet$] Lie Groups, Homogenous Spaces, Symmetric Spaces
\item[$\bullet$] Riemannian Manifolds $({\mathscr M}, g)$ 
\item[$\bullet$] Smooth Metric Measure Spaces $({\mathscr M}, g, d\mu)$
\item[$\bullet$] Metric Measure Spaces $(X, \mathsf d, \mathfrak m)$
(No differentiable structure!)
\end{itemize}

\section{Smooth metric measure spaces $({\mathscr M}, g, d\mu)$}

A smooth metric measure space is a triple $({\mathscr M}, g, d\mu)$ where $({\mathscr M}, g)$ 
is a smooth Riemannian manifold of dimension $n \ge 2$, $d\mu=e^{-\varphi} dv_g$ is a positive 
weighted measure and $dv_g$ is the Riemannian volume measure on ${\mathscr M}$. 
The diffusion operator associated with the triple $(\mathscr M,g,d\mu)$ 
is the $\varphi$-Laplacian (that depending on the context is also called the Witten, drifting 
or weighted Laplacian). It is defined for $w \in \mathscr{C}^2(\mathscr M)$ by 
\begin{equation} \label{f-Lap-definition}
\Delta_\varphi w = \Delta w - \langle \nabla \varphi, \nabla w\rangle.  
\end{equation} 
As is readily seen it is a symmetric diffusion operator with respect to the invariant measure $d\mu$ 
and in the Riemannian context $\varphi \equiv 0$ (i.e., $d\mu=dv_g$) coincides with the usual 
Laplace-Beltrami operator.

\qquad \\
{\bf The Ricci tensors ${\mathscr Ric}_\varphi^m(g)$ and the $\varphi$-Laplacian $\Delta_\varphi$.} 
The Bakry-\`Emery $m$-Ricci curvature tensor ${\mathscr Ric}^m_\varphi(g)$ associated with the 
triple $({\mathscr M}, g, d\mu)$ is defined as
\begin{equation} \label{Ricci-m-f-intro}
{\mathscr Ric}^m_\varphi(g) = {\mathscr Ric}(g) + \nabla\nabla \varphi - \frac{\nabla \varphi \otimes \nabla \varphi}{m-n}, \qquad m \ge n. 
\end{equation}
Here ${\mathscr Ric}(g)$ is the usual Riemannain Ricci curvature tensor of $g$, $\nabla \nabla \varphi={\rm Hess}(\varphi)$ 
denotes the Hessian of $\varphi$, and $m \ge n$ is a constant ({\it see} \cite{Bak}). For the sake of clarification, in 
\eqref{Ricci-m-f-intro}, when $m=n$, by convention $f$ is only allowed to be a constant, thus giving 
${\mathscr Ric}^n_\varphi(g)={\mathscr Ric}(g)$. By formally taking the limit $m \nearrow \infty$ in 
\eqref{Ricci-m-f-intro} we also define
 \begin{equation} \label{Ricci-f-intro}
{\mathscr Ric}_\varphi(g)={\mathscr Ric}(g) + \nabla\nabla \varphi. 
\end{equation} 
The significance of \eqref{Ricci-f-intro} lies in its appearance in the weighted Bochner-Weitzenb\"ock formula, asserting 
that for any function $w$ of class ${\mathscr C}^3(M)$ it holds,   
\begin{equation} \label{Bochner-1}
\frac{1}{2} \Delta_\varphi |\nabla w|^2 - \langle \nabla w, \nabla \Delta_\varphi w \rangle
= |\nabla\nabla w|^2 + {\mathscr Ric}_\varphi (\nabla w, \nabla w).
\end{equation}
The above formula does not have a counterpart for $m<\infty$, however, by an application of Cauchy-Schwarz inequality, 
one can easily show that the following {\it inequality} holds for 
${\mathscr Ric}_\varphi^m(g)$, 
\begin{equation} \label{Bochner-2}
\frac{1}{2} \Delta_\varphi |\nabla w|^2 - \langle \nabla w, \nabla \Delta_\varphi w \rangle
\ge \frac{1}{m} |\Delta_\varphi w|^2 + {\mathscr Ric}_\varphi^m (\nabla w, \nabla w).
\end{equation}
Both these relations and the corresponding Ricci tensors will appear in the formulation 
of the curvature-dimension condition below.  (See also \cite{AGS, Bak, Taheri-GE-1, Taheri-GE-2, TVahNA, TVahCurv} for more.)

\qquad \\
{\bf Some examples of smooth metric measure spaces $({\mathscr M},g,d \mu)$ with $d\mu=e^{-\varphi}dv_g$.} 
Let us proceed by giving some related and useful examples on smooth metric measure spaces to illustrate 
the scope of the theory.

\begin{itemize}
\item (Riemannian) $({\mathscr M}, g)$ with constant $\varphi$: 
$\Delta_\varphi \equiv \Delta$, ${\mathscr Ric}_\varphi^m(g)={\mathscr Ric}_\varphi(g)={\mathscr Ric}(g)$.
\item (Euclidean weighted space) $({\mathbb R}^n, g_{{\mathbb R}^n}, d\mu = e^{-\varphi} dx)$: 
$\Delta_\varphi u = \Delta u - \langle \nabla \varphi, \nabla u \rangle $. 
Here ${\mathscr Ric}_\varphi(g) = {\mathscr Ric}(g) + \nabla \nabla \varphi 
=  \nabla \nabla \varphi$ and so ${\mathscr Ric}_\varphi(g) \ge 0 \iff \varphi$ is convex. 
Furthermore from \eqref{Ricci-m-f-intro} we have 
\begin{align}
{\mathscr Ric}^m_\varphi(g) 
= \nabla \nabla \varphi  - \frac{\nabla \varphi \otimes \nabla \varphi}{m-n}, \qquad n \le m <\infty. 
\end{align}
\item (Ornstein-Uhlenbeck) $({\mathbb R}^n, g_{{\mathbb R}^n}, d\gamma_n)$ with $d\gamma_n$ as in \eqref{Gaussian-measure}: 
$\Delta_\varphi u = \Delta u - \langle x, \nabla u \rangle$. Here 
${\mathscr Ric}_\varphi(g) = \nabla \nabla \varphi = {\rm I}$. Moreover,  
\begin{align}
{\mathscr Ric}^m_\varphi(g) = \nabla \nabla \varphi - \frac{\nabla \varphi \otimes \nabla \varphi}{m-n} 
= {\rm I} - \frac{x \otimes x}{m-n}, \qquad n \le m <\infty. 
\end{align}
\item (Interval case) ${\mathscr M}= (a, b)$, $d\mu = e^{-\varphi} dx$: $\Delta_\varphi u = u''-\varphi'u'$. 
Here it is seen that 
\begin{align}
{\mathscr Ric}^m_\varphi(g) \ge k, \qquad {\text when}  \qquad \varphi'' \ge k + \varphi'^2/(m-1).
\end{align}
The cases $d\mu = (\sin t)^{m-1} \, dt$ on $(0, \pi)$ and $d\mu = (\sinh t)^{m-1} \, dt$ on $(0, \infty)$ 
arise in the study of isoperimetric inequality on ${\mathbb S}^n$ and ${\mathbb H}^n$ 
under ${\rm CD}(\pm 1,m)$ for $m \ge n$ respectively.
\item (Jacobi polynomials: $y ={\mathscr P}^{(\alpha, \beta)}_k(t)$ with $\alpha, \beta>-1, k \ge 0$) 
Here ${\mathscr M}=(-1,1)$ and $d\mu= (1-t)^\alpha (1+t)^\beta dt$, 
\begin{equation}
\left( 1-t^{2}\right) \frac{d^2y}{dt^2} - (\alpha-\beta+(\alpha+\beta+2)t) \frac{dy}{dt} + k(k+\alpha+\beta+1) y =0. 
\end{equation}
\item Orthogonal Polynomials (Bessel, Hermite, Jacobi, Laguerre, Tchebychev, Ultra-spherical, Legendre, 
Hypergeometric series, etc.) and Sturm-Liouville systems.
\end{itemize}

\qquad \\
{\bf The curvature-dimension condition ${\rm CD}(k, m)$.} 
In the context of smooth metric measure spaces one imposes Ricci curvature lower bounds in the form 
${\mathscr Ric}_\varphi(g) \ge k g$ or the stronger form ${\mathscr Ric}^m_\varphi(g) \ge k g$ for some fixed 
constants $k \in {\mathbb R}$ and $n \le m<\infty$. This enables one to use various comparison results (e.g., 
Hessian and Laplace comparison, Bishop-Gromov volume comparison, {\it etc.}) that are crucial instruments in 
doing analysis and geometry. 

A curvature lower bound in the form ${\mathscr Ric}^m_\varphi(g) \ge k g$ (with $k \in {\mathbb R}$) 
implies that the diffusion operator $L=\Delta_\varphi$ satisfies the curvature-dimension 
condition ${\rm CD}(k,m)$. To elaborate on this last point 
further, we first recall the definition of the carre du champ operator $\Gamma[L]$ 
associated with a Markov diffusion operator $L$. This is given by 
\begin{equation} \label{Gamma-One-L-eq}
\Gamma[L] (u,v) = \Gamma_1[L] (u,v) := \frac{1}{2} [L(uv) - u L v - v Lu]. 
\end{equation}

Higher order iterates are then defined inductively by replacing the product operation with $\Gamma$ respectively. 
Indeed, the second order iterated carre du champ operator $\Gamma_2[L]$, can be seen to be, 
\begin{equation} \label{Gamma-Two-L-eq}
\Gamma_2[L] (u,v) := \frac{1}{2} [ L \Gamma(u,v) - \Gamma (u, Lv) - \Gamma (Lu, v)]. 
\end{equation}

For $u=v$ we write $\Gamma[L](u)=\Gamma[L](u,u)$ and $\Gamma_2[L](u)=\Gamma_2[L](u,u)$. 
Now the Markov diffusion operator $L$ is said to satisfy the Bakry-\'Emery curvature-dimension 
condition ${\rm CD}(k,m)$ {\it iff} for some $m \ge 2$ and $k \in {\mathbb R}$, 
\begin{equation} \label{CDkm-definition-equation}
\Gamma_2[L](w) =  \frac{1}{2} L \Gamma(w) - \Gamma (w, Lw) \ge \frac{1}{m} (Lw)^2 + k \Gamma[L](w). 
\end{equation}

By a direct calculation, it is seen that for $L=\Delta_\varphi$ as in \eqref{f-Lap-definition} 
and $\Gamma[L]$ and $\Gamma_2[L]$ as in \eqref{Gamma-One-L-eq} and \eqref {Gamma-Two-L-eq}, we have the first 
and second order relations 
\begin{equation} \label{Un-iterated-CDC}
\Gamma[L] (w) = |\nabla w|^2,
\end{equation}
\begin{equation} \label{iterated-CDC}
\Gamma_2 [L] (w) = \frac{1}{2} L |\nabla w|^2 - \langle \nabla w, \nabla L w \rangle, 
\end{equation} 
respectively. Hence, in light of \eqref{Bochner-2}, \eqref{Un-iterated-CDC} and \eqref{iterated-CDC} it follows from the 
curvature lower bound ${\mathscr Ric}^m_\varphi(g) \ge k g$ that here the 
iterated carre du champ operator $\Gamma_2[L]$ satisfies the inequality 
\begin{align} \label{CD-mk}
\Gamma_2[L](w) 
&\ge \frac{1}{m} (L w)^2 + {\mathscr Ric}^m_\varphi (\nabla w, \nabla w) \nonumber \\
&\ge \frac{1}{m} (L w)^2 + k |\nabla w|^2 = \frac{1}{m} (Lw)^2 + k \Gamma[L](w). 
\end{align}
In contrast, a curvature lower bound in the form ${\mathscr Ric}_\varphi(g) \ge k g$ implies the curvature-dimension condition 
${\rm CD}(k,\infty)$, in the sense that by virtue of \eqref{Bochner-1}, \eqref{Un-iterated-CDC} and \eqref{iterated-CDC} one only has
\begin{align} \label{CD-inftyk}
\Gamma_2[L](w) 
&= |\nabla \nabla w|^2 + {\mathscr Ric}_\varphi (\nabla w, \nabla w) 
\ge k |\nabla w|^2 = k \Gamma[L](w). 
\end{align}

Curvature lower bounds and/or curvature-dimension conditions are instrumental for doing analysis 
and geometry on smooth metric measure spaces. Below we see their significance in Poincar\'e and 
logarithmic Sobolev inequalities. First let us look at some equivalent versions of the conditions 
formulated in terms of the semigroup $P_t = e^{t\Delta_\varphi}$. Below we assume that $d\mu$ 
is a probability measure and throughout we take $f>0$.

\begin{theorem}
The following conditions for $k>0$ are equivalent: 
\begin{itemize}

\item[$\rhd$] $($Curvature-Dimension$)$ ${\rm CD}(k, \infty)$.
\footnote{See \eqref{CDkm-definition-equation}--\eqref{CD-inftyk}, in particular, the equivalence with ${\mathscr Ric}_\varphi(g) \ge kg$. 
On ${\mathbb R}^n$ this is the {\it uniform convexity} of $\varphi$, i.e., ${\mathscr Ric}_\varphi(g) \ge kg \iff \nabla \nabla \varphi \ge k {\rm I}$.}
\item[$\rhd$] $($Commutation$)$ $|\nabla P_t f| \leq e^{-kt} P_t |\nabla f|$.
\item[$\rhd$] $($Contraction in Wasserstein Space$)$ 
\begin{equation}
W_2(P_t f d\mu, P_t g d\mu) \le e^{-kt} W_2(f d\mu, g d\mu).
\end{equation} 
\end{itemize}
\end{theorem}

\begin{theorem} \label{thm-Poincare-CD}
${\rm CD}(k, \infty) \implies$ Poincaré inequality with constant $1/k$.
\end{theorem}

\begin{proof}
Note that Poincaré inequality (spectral gap) amounts to showing the inequality 
${\rm Var}_\mu(f) \le (1/\lambda) {\rm E}_\mu(f)$ where the smallest such 
$\lambda>0$ is the first (non-zero) eigenvalue of $-\Delta_\varphi$ on ${\mathscr M}$. Thus 
we write 
\begin{align} \label{PI-SMMS}
{\rm Var}_\mu(f) &= \int_{\mathscr M} f^2 \, d \mu - \left(\int_{\mathscr M} f \, d \mu \right)^2 
= \int_{\mathscr M} \left|f - \int_{\mathscr M} f \, d\mu \right|^2 \, d\mu \nonumber \\
&= - \int^{\infty}_0 \frac{d}{dt} \int_{\mathscr M} [P_t f]^2  \, d\mu \, dt 
= - 2 \int^{\infty}_0 \int_{\mathscr M} [\Delta_\varphi P_t f] [P_t f] \, d \mu \, dt \nonumber \\
&= 2 \int^{\infty}_0 \int_{\mathscr M} |\nabla P_t f|^2 \, d \mu \, dt 
\le 2 \int^{\infty}_0 \int_{\mathscr M} e^{-2kt} P_t |\nabla f|^2 \, d \mu \, dt \nonumber \\
&\le \frac{1}{k} \int_{\mathscr M} |\nabla f|^2 \, d\mu = \frac{1}{k} {\rm E}_\mu(f)
\end{align}
as required.
\end{proof}

\begin{remark}{\em The Poincar\'e inequality ${\rm Var}_\mu(f)  \le (1/k) {\rm E}_\mu(f)$ 
gives exponential decay in {\it variance}: ${\rm Var}_\mu (P_t f) \le e^{-2kt} {\rm Var}_\mu(f)$. 
Note that the curvature-dimension condition ${\rm CD}(k,m)$ gives a Poincar\'e inequality with 
constant $(m-1)/(mk)$.
}
\end{remark}

\begin{theorem}
${\rm CD}(k, \infty) \implies$ Logarithmic Sobolev inequality with constant $1/k$. 
\end{theorem}

\begin{proof}
The claim amounts to showing ${\rm Ent}_\mu (f^2) \le 1/(2k) I_\mu(f^2) = (2/k) {\rm E}_\mu (f)$. Thus we write  
\begin{align}
{\rm Ent}_\mu (f) 
&= \int_{\mathscr M} f \log f \, d \mu - \int_{\mathscr M} f \, d \mu \log \int_{\mathscr M} f \, d \mu \nonumber \\
&= - \int^{\infty}_0 \frac{d}{dt} \int_{\mathscr M} [P_t f \log P_t f] \, dt \, d \mu 
= - \int^{\infty}_0 \int_{\mathscr M} [\Delta_\varphi P_t f] [\log P_t f] \, dt \, d \mu \nonumber \\
&=  \int_{\mathscr M} \int^{\infty}_0 \frac{|\nabla P_t f|^2}{P_t f} \, dt \, d \mu 
\le  \int_{\mathscr M} \int^{\infty}_0 e^{-2kt} \frac{[P_t |\nabla f|]^2}{P_t f} \, dt \, d \mu.
\end{align}
Now by Cauchy-Schwarz inequality 
\begin{align}
[P_t |\nabla f|]^2 = \left[ P_t \left( \sqrt f \frac{|\nabla f|}{\sqrt f} \right) \right]^2  \le P_t f \times P_t \left[\frac{|\nabla f|^2}{f} \right].
\end{align}
This yields 
\begin{align} \label{LSI-SMMS}
{\rm Ent}_\mu (f)
&= \int_{\mathscr M} f \log f \, d \mu - \int_{\mathscr M} f \, d \mu \log \int_{\mathscr M} f \, d \mu \nonumber \\
&\le \int_{\mathscr M} \int^{\infty}_0 e^{-2k t} P_t \left[ \frac{|\nabla f|^2}{f} \right] \, dt \, d \mu
\le \frac{1}{2k} \int_{\mathscr M} \frac{|\nabla f|^2}{f} \, d \mu = \frac{1}{2k} I_\mu (f),
\end{align}
as required. 
\end{proof}

\begin{remark}{\em Logarithmic Sobolev inequatiy \eqref{LSI-SMMS} gives exponential decay 
in entropy: ${\rm Ent}_\mu (P_t f) \le e^{-2kt} {\rm Ent}_\mu(f)$.}
\end{remark}

\begin{theorem}
${\rm CD}(k, \infty) \implies$ Isoperimetric inequality $\sqrt k I(\mu(A)) \le \mu^+ (A)$ where 
$I$ is the isoperimetric profile of the Gaussian $[$as in \eqref{Gaussian-isoperimetric-inequality}$]$.
\end{theorem}

\section{Nonlinear diffusion equations on smooth metric measure spaces}

In this last section we look more closely at (nonlinear) diffusion equations on smooth metric measure spaces. 
Here we confine to the porous medium and the fast diffusion equations (\cite{TV-PME-a}-\cite{TV-FDE-a}). 
\footnote{Diffusion equations have a close affinity with geometric and functional inequalities. 
One strong connection comes from the study of entropies and gradient flows in Wasserstein spaces. 
In a related paper in this volume (by V.~Vahidifar) the nonlinear heat equation $\square u(x,t) 
= \partial_t u (x,t) - \Delta_\varphi u (x,t) = \mathscr N(t,x,u(x,t))$ and their gradient estimates 
on smooth metric measure spaces are discussed.} 
For a (space-time) function $u=u(x,t)>0$ and fixed exponent $p>0$ this can be written in the form  
\begin{align} \label{eq11}
\square_p u(x,t) = \partial_t u (x,t) - \Delta_\varphi u^p (x,t) = \mathscr N(t,x,u(x,t)).   
\end{align}
Here ${\mathscr N}={\mathscr N}(t,x,u)$ is a sufficiently smooth nonlinearity depending on the space-time variable $(x,t)$ 
and the dependent variable (solution) $u$. The porous medium equation corresponds to the range $p>1$ and the fast 
diffusion equation to the range $0<p<1$.

We write ${\mathscr B}_R(x_0)$ for the closed geodesic ball in 
${\mathscr M}$ centered at $x_0$ with radius $R>0$. Likewise we write 
$Q_{R,T}(x_0)$ for the compact parabolic space-time cylinder with 
lower base ${\mathscr B}_R(x_0) \times \{t_0-T\}$ for $t_0 \in {\mathbb R}$ 
and height $T>0$, specifically, $Q_{R,T}(x_0)={\mathscr B}_{R}(x_0) \times [t_0-T, t_0] \subset 
{\mathscr M} \times (-\infty, \infty)$.
When the choice $(x_0, t_0)$ is clear from the context we write ${\mathscr B}_R$, $Q_{R, T}$. 
We write $X_+=\max(X, 0)$ and $X_-=\max(-X, 0)$.

\qquad \\
{\bf The nonlinear porous medium equation.} 
The first nonlinear diffusion equation we look at here is the one of slow diffusion type (or equivalently porous medium type). 
This is \eqref{eq11} with $p>1$. We also set 
\begin{align}\label{EQ-eq-2.4}
\Sigma(t,x,v)= p [(p-1)v/p ]^{\frac{p -2}{p-1}} 
\mathscr N \left(t,x,[(p-1)v/p]^{\frac{1}{p-1}}\right).
\end {align}

\begin{theorem} \label{thm1-EQ-PME-static}
Let $({\mathscr M}, g, d\mu)$ be a complete smooth metric measure space with $d\mu=e^{-\varphi} dv_g$ and assume 
${\mathscr Ric}_\varphi^m (g) \ge -(m-1)k g$ in ${\mathscr B}_R$ for some $k \ge 0$, $m \ge n$ 
and $R>0$. Let $u$ be a positive solution 
to \eqref{eq11} with $1 < p < 1+1/\sqrt{m-1}$ and $v =pu^{p-1}/(p-1)$ and $M= \sup_{Q_{R,T}} v$. 
Then there exists $C=C(p,m)>0$ such that for every $(x,t)$ in $Q_{R/2,T}$ with $t>t_0-T$ we have 
\begin{align}\label{eq-2.1-EQ-static}
v^{\frac{1}{2(p-1)}}|\nabla v| (x,t) 
\le C \left \{ \begin {array}{ll}
\left[ \dfrac{k^{1/4}}{\sqrt R} + \dfrac{1}{R} + \sqrt k \right] M^{1+\frac{1}{2(p-1)}}
\\
\\
+ \sup_{Q_{R, T}}\left\{v^{p/[2(p-1)]} \left[2 \Sigma_v(t,x,v)+\dfrac{\Sigma(t,x,v)}{(p-1)v} \right]_+^{1/2}\right\}
\\
\\
+ \dfrac{M^\frac{p}{2(p-1)}}{\sqrt {t-t_0+T}} 
+\sup_{Q_{R, T}} \left\{ \left[ v^{(2p+1)/[2(p-1)]} |\Sigma_x(t,x,v)| \right]^{1/3} \right\} 
\end{array}
\right\}.
\end{align}
\end{theorem}

\begin{theorem} \label{ancient-2}
Let $({\mathscr M}, g, d\mu)$ be a complete smooth metric measure space with $d\mu=e^{-\varphi}dv_g$ and 
${\mathscr Ric}^m_\varphi(g) \ge 0$. Assume $(3-2p) {\mathscr N}(u) - 2u {\mathscr N}_u(u) \ge 0$ 
for all $u>0$ where $1< p < 1+1/\sqrt{m-1}$. Then any positive ancient solution to the nonlinear porous 
medium equation 
\begin{equation} \label{ancient-equation-PME}
\square_p u(x,t) = \partial_t u (x,t) - \Delta_\varphi u^p (x,t) = {\mathscr N}(u(x,t)), 
\end{equation}
satisfying the growth at infinity $u(x,t) = o([\varrho(x) + \sqrt{|t|}]^{2/(2p-1)})$, 
must be spatially constant. In particular, if additionally, ${\mathscr N}(u) \ge a$ for some $a>0$ and all 
$u>0$ then \eqref{ancient-equation-PME} admits no such ancient solutions. 
\end{theorem}

\begin{theorem} \label{first-cor-last}
Let $(\mathscr M,g,d\mu)$ be a smooth metric measure space with $\mathscr M$ 
closed and $d\mu=e^{-\varphi} dv_g$ and let ${\mathscr Ric}^m_\varphi(g) \ge - (m-1) k g$ 
on ${\mathscr M}$ for some $k \ge 0$. 
Let $u$ be a positive smooth solution to \eqref{eq11} where 
$1<p \le 1+1/\sqrt{m-1}$ and set $v =pu^{p-1}/(p-1)$. Let $\Gamma=\Gamma(v)$ 
with $v>0$ be a given function of class $\mathscr{C}^2$ and assume the following 
conditions hold for some $a$: 
\begin{itemize}
\item $\Gamma'(v) \Sigma(t,x,v) \le 0$,
\item $\Gamma'(v) + v \Gamma''(v) \ge 0$,
\item $\langle \nabla v, \Sigma_x(t,x,v) \rangle \le 0$,  
\item $\Sigma_v(t,x,v) + \Sigma(t,x,v)/[2(p-1)v] \le a$.
\end{itemize}
Then for every $x \in \mathscr M$ and $0<t \le T$ we have the gradient bound
\begin{equation}
[v^\frac{1}{p-1} |\nabla v|^2] (x,t) 
\le e^{2(MK+a)t} \left\{ \max_{\mathscr M} 
\left[ v^\frac{1}{p-1} |\nabla v|^2 + \Gamma(v) \right]_{t=0} - \Gamma(v(x,t)) \right\},
\end{equation}
where $M=\sup v$ and $K=(p-1)(m-1) k$.
\end{theorem}

\begin{theorem} \label{second-cor-last}
Let $(\mathscr M,g,d\mu)$ be a smooth metric measure space with $\mathscr M$ 
closed and $d\mu=e^{-\varphi} dv_g$ and let ${\mathscr Ric}^m_\varphi(g) \ge - (m-1) k g$ 
on ${\mathscr M}$ for some $k \ge 0$.
Let $u$ be a positive smooth solution to \eqref{eq11} 
with $1<p \le 1+1/\sqrt{m-1}$ and set $v =pu^{p-1}/(p-1)$. 
Assume the following conditions hold:
\begin{itemize}
\item $\Sigma(t,x,v) \le 0$,
\item $\langle \nabla v, \Sigma_x(t,x,v) \rangle \le 0$, 
\item $\Sigma_v(t,x,v) + \Sigma(t,x,v)/[2(p-1)v] \le 0$. 
\end{itemize}
Then for every $x \in \mathscr M$ and $0<t \le T$ we have the gradient bound
\begin{equation}
\frac{p^2t}{1+2MKt} [v^\frac{1}{p-1} |\nabla v|^2](x,t) \le (p-1) 
\left[ \max_{\mathscr M} v^\frac{p}{p-1}(x,0) - v^\frac{p}{p-1}(x,t) \right],
\end{equation}
where $M=\sup v$ and $K=(p-1)(m-1) k$.
\end{theorem}

\qquad \\
{\bf The nonlinear fast diffusion equation.} This is \eqref{eq11} with $0<p<1$. 
For technical reasons we introduce $p_0=p_0(m)>0$ as the larger of the values $1/2$ 
and $1-1/\sqrt{m-1}$ ($m \ge 2)$, that is, $p_0=1/2$ when $2 \le m \le 5$ and 
$p_0=1-1/\sqrt{m-1}$ when $m \ge 5$. Thus we always have $p_0(m) \ge 1/2$, and when $m \ge 4$ 
we have $p_0 \le p_c=1-2/m$. We also set 
\begin{align} \label{Sigma-Star-N-definition}
\Sigma^\star(t, x, v) = (p -1/2) v^{1-\frac{1}{p-1/2}} \mathscr N (t,x,v^{\frac{1}{p-1/2}}).
\end{align}

\begin{theorem} \label{thm-6.1-EQ-FDE-static}
Let $({\mathscr M}, g, d\mu)$ be a complete smooth metric measure space with $d\mu=e^{-\varphi} dv_g$ and  
${\mathscr Ric}_\varphi^m (g) \ge -(m-1)kg$ in ${\mathscr B}_R$ for some $k \ge 0$, $m \ge n$ and $R>0$. 
Let $u$ be a positive solution to \eqref{eq11} with $p_0< p < 1$, $v = u^{p-1/2}$ and 
$M= \sup_{Q_{R,T}} v$. Then there exists $C=C(p,m)>0$ such that for $(x,t)$ in 
$Q_{R/2,T}$ with $t>t_0-T$ we have 
\begin{align} \label{thm-6.1-EQ-FDE-static-equation}
|\nabla v|(x,t)
\le C \left \{ \begin {array}{ll}
\dfrac{M^\frac{p}{2p-1}}{\sqrt {t-t_0+T}} 
+\sup_{Q_{R, T}} \left\{v^{\frac{2p}{3(2p-1)}} |\Sigma^\star_x(t,x,v)|^{1/3}\right\} 
\\
\\
+ \left[ \dfrac{k^{1/4}}{\sqrt R} + \dfrac{1}{R} + \sqrt k \right] M
+ \sup_{Q_{R, T}}\left\{v^{\frac{p}{2p-1}} [\Sigma^\star_v(t,x,v)]_+^{1/2}\right\}
\end{array}
\right\}.
\end{align}
\end{theorem}

\begin{theorem} \label{ancient-2}
Let $({\mathscr M}, g, d\mu)$ be a complete smooth metric measure space with $d\mu=e^{-\varphi}dv_g$ and 
${\mathscr Ric}^m_\varphi(g) \ge 0$. Assume $(3-2p) {\mathscr N}(u)/u - 2 {\mathscr N}_u(u) \ge 0$ 
for all $u>0$ where $p_0<p<1$. Then any positive ancient solution to the nonlinear 
fast diffusion equation 
\begin{equation} \label{ancient-equation-2}
\square_p u (x,t) = \partial_t u (x,t) - \Delta_\varphi u^p (x,t) = {\mathscr N}(u(x,t)), 
\end{equation}
satisfying the growth at infinity $u(x,t) = o([\varrho(x) + \sqrt{|t|}]^{2/p})$ 
must be spatially constant. If, in addition, ${\mathscr N}(u) \ge a$ for some $a>0$ and 
all $u>0$ then \eqref{ancient-equation-2} admits no such ancient solutions. 
\end{theorem}

\begin{theorem} \label{first-cor-last-last}
Let $(\mathscr M,g,d \mu)$ be a smooth metric measure space with $\mathscr M$ closed 
and $d\mu=e^{-\varphi} dv_g$ and let ${\mathscr Ric}^m_\varphi(g) \ge - (m-1) k g$ 
on ${\mathscr M}$ for some $k \ge 0$. 
Let $u$ be a positive smooth solution to \eqref{eq11} where $p_0<p<1$ and set $v = u^{p-1/2}$. 
Let $\Gamma=\Gamma(v)$ with $v>0$ be a given function of class $\mathscr{C}^2$ and assume 
the following conditions hold for some $a$: 
\begin{itemize}
\item $\Sigma^\star(t,x,v) \le a$, 
\item $\Gamma'(v) \Sigma^\star (t,x,v) \le 0$,
\item $\Gamma'(v) + v \Gamma''(v) \ge 0$,
\item $\langle \nabla v, \Sigma^\star_x(t,x,v) \rangle \le 0$.  
\end{itemize}
Then for all $x \in \mathscr M$ and $0<t \le T$ we have 
\begin{equation}
|\nabla v|^2 (x,t) 
\le e^{2(MK+a)t} \left\{ \max_{\mathscr M} 
\left[ |\nabla v|^2 + \Gamma(v) \right]_{t=0} - \Gamma(v(x,t)) \right\},
\end{equation}
where $M=\sup v$ and $K=(p-1)(m-1) k$.
\end{theorem}

\begin{theorem} \label{second-cor-last-last}
Let $(\mathscr M,g,d\mu)$ be a smooth metric measure space with $\mathscr M$ closed 
and $d\mu=e^{-\varphi} dv_g$ and let ${\mathscr Ric}^m_\varphi(g) \ge - (m-1) k g$ 
on ${\mathscr M}$ for some $k \ge 0$. Let $u$ be a positive smooth solution to \eqref{eq11} 
where $p_0<p <1$ and set $v =u^{p-1/2}$. Assume the following conditions hold: 
\begin{itemize}
\item $\Sigma^\star(t,x,v) \le 0$,
\item $\Sigma^\star_v(t,x,v) \le 0$,
\item $\langle \nabla v, \Sigma^\star_x(t,x,v) \rangle \le 0$. 
\end{itemize}
Then for $x \in \mathscr M$ and $0<t \le T$ we have 
\begin{equation}
\frac{t |\nabla v(x,t)|^2}{1+2MKt} \le \frac{(2p-1)^2}{8p^3} \left[ \max_{\mathscr M} v^\frac{2p}{2p-1}(x,0) - v^\frac{2p}{2p-1}(x,t) \right], 
\end{equation}
where $M=\sup v$ and $K=(p-1)(m-1) k$.
\end{theorem}

\qquad \\
{\bf Acknowledgement.} I thanks Prof Michael Ruzhansky for the opportunity 
to visit the Ghent Analysis and PDE centre during Summer 2024. I also wish to thank the organisers of the 
workshop ``Inverse Problems and Application" for organising such a nice, timely and stimulating meeting. 
Support from the Engineering and Physical Sciences Research Council (EPSRC) through the grant 
EP/V027115/1 is gratefully acknowledged.

\end{document}